\documentclass[11pt]{amsart}
\usepackage{graphicx}
\usepackage[active]{srcltx}
 \makeatletter
\renewcommand*\subjclass[2][2000]{%
  \def\@subjclass{#2}%
  \@ifundefined{subjclassname@#1}{%
    \ClassWarning{\@classname}{Unknown edition (#1) of Mathematics
      Subject Classification; using '1991'.}%
  }{%
    \@xp\let\@xp\subjclassname\csname subjclassname@#1\endcsname
  }%
}
 \makeatother
\usepackage{enumerate,amssymb,  mathrsfs,yhmath}

\newtheorem{theorem}{Theorem}[section]
\newtheorem{lemma}[theorem]{Lemma}
\newtheorem*{lemma*}{Lemma}
\newtheorem{proposition}[theorem]{Proposition}
\newtheorem{corollary}[theorem]{Corollary}

\def\1ton{1,2,\ldots,n}
\def\det{{\rm det}}

\usepackage{amssymb}
\usepackage{amsthm}
\usepackage{mathrsfs, amsfonts, amsmath}
\usepackage{graphicx}

\newcommand{\bydef}{\stackrel{{\rm def}}{=\!\!=}}

\newcommand{\onto}{\xrightarrow[]{{}_{\!\!\textnormal{onto}\!\!}}}

\newcommand{\A}{\mathbb{A}}

\newcommand{\X}{\mathbb{X}}

\newcommand{\Y}{\mathbb{Y}}

\theoremstyle{definition}

\theoremstyle{remark}
\newtheorem{remark}{Remark}[section]

\numberwithin{equation}{section}

\newcommand{\abs}[1]{\lvert#1\rvert}

\def\XXint#1#2#3{{\setbox0=\hbox{$#1{#2#3}{\int}$}
\vcenter{\hbox{$#2#3$}}\kern-.5\wd0}}

\def\ge{\geqslant}
\begin{document}

\title[Neohookean deformations of annuli in the Euclidean space]{Neohookean deformations of annuli in the higher dimensional Euclidean space} \subjclass{Primary 31A05;
Secondary 42B30 }


\author{David Kalaj}
\address{University of Montenegro, Faculty of Natural Sciences and
Mathematics, Cetinjski put b.b. 81000 Podgorica, Montenegro}
\email{davidk@ac.me}

\author{Jian-Feng Zhu}
\address{Jian-Feng Zhu, Department of Mathematics, Shantou University, Shantou, Guangdong 515063,
People's Republic of China
and School of Mathematical Sciences,
Huaqiao University,
Quanzhou 362021, People's Republic of China.
} \email{flandy@hqu.edu.cn}

\begin{abstract}
Let $n>2$ be an integer and assume that
$\mathbb{A}=\{x\in\mathbf{R}^n:1<|x|<R\}$ and
$\A_\ast = \{y  \in  \mathbf{R}^n:  1   < |y| <  R_\ast\}$ be two
annuli in  Euclidean space $\mathbf{R}^n$.
Assume that $\mathcal{F}(\A, \A_\ast)$ (resp. $\mathcal{R}(\A, \A_\ast)$) be the  class  of  all  orientation  preserving (resp. radial) homeomorphisms
$h  : \A \mapsto \A_\ast$ in the Sobolev space $\mathcal{W} ^{1,n}(\A, \A_\ast)$ which keep the boundary circles in the same order.
In this paper, we extended the corresponding results of  Iwaniec and Onninen which was published in {\it Math. Ann.} Vol. 348, 2010.
\end{abstract}

\maketitle
\section{Introduction}
Let
$\mathbb{X}=\{x\in\mathbf{R}^2: r<|x|<R\}$ and  $\mathbb{Y}=\{y\in\mathbf{R}^2: r_*<|y|<R_*\}$
be two concentric annuli of the complex plane.
The mapping problem between $\mathbb{X}$ and $\mathbb{Y}$ by means of harmonic diffeomorphisms raised the J. C. C. Nitsche conjecture, which states that
there is a harmonic diffeomorphism between $\X$ and $\Y$ if and only if
\begin{equation}\label{niin}\frac{R_*}{r_*}\ge \frac{1}{2}\left(\frac{R}{r}+\frac{r}{R}\right)\end{equation}
(see \cite[Theorem 1.4]{nconj}). The theorem can be related to the minimal surfaces. On the other hand the inequality \eqref{niin} is important for the existence of diffeomorphic minimizers of Dirichlet's energy between annuli on the plane  \cite{astala}. For certain generalizations we refer to \cite{memoirs, kalajarxiv, klondon}. It should be noted that some results have been obtained for general doubly connected domains $\X$ and $\Y$ in complex plane and in Riemann surfaces subject to the condition $\mathrm{mod}(\X)\le \mathrm{mod}(\Y)$ (\cite{invent, calculus}).

In \cite{annalen}, Iwaniec and Onninen studied the neohookean energy for the so-called deformations between two annuli in the Euclidean complex plane.
They studied a concrete extremal problem motivated by recent remarkable
relations between Geometric Function Theory  (mappings of finite distortion)
and the Theory of Nonlinear Elasticity (hyperelastic deformations in particular). Both theories are governed by variational principles.

This paper continues to study the same problem, in the space and in the plane but under slightly different circumstances. Here we consider  deformations of bounded spatial annular  domains $\A(r, R)=\{x\in\mathbf{R}^n:r<|x|<R\}$ and  $\A_\ast(r_*, R_*)=\{y\in\mathbf{R}^n: r_\ast<|y|<R_\ast\}$. Let a homogeneous isotropic elastic body in the reference configuration $\A(r, R)$ be deformed into configuration $\A_\ast(r_*, R_*)$.
The general law of hyperelasticity tells us that there exists a stored energy function
$E : \mathbf{R}_+ \times \mathbf{R}_+ \to \mathbf{R}$ that characterizes the elastic and mechanical properties of the
material. The subject of the investigation are orientation preserving homeomorphisms
$h : \A(r, R) \onto
\A_\ast(r_*, R_*)$ of smallest energy;
$$\mathcal{E}_\Phi=\mathcal{E}[h]\bydef \int_{\X}\|Dh\|^n +\Phi(\det Dh)dx$$
called extremal deformations.  Here $\Phi:(0,+\infty)\to (0,+\infty)$ is a positive, convex and three times differentiable function. This is the several dimensional generalisation of the planar case considered by Iwaniec and Astala in \cite{annalen}.

We  consider  the  class $\mathcal{F}(\A, \A_\ast)$ (resp. $\mathcal{R}(\A, \A_\ast)$)  of  all  orientation  preserving (resp. radial) homeomorphisms
$h:\A(r,R)\to \A(r_\ast, R_\ast)$ in the Sobolev space $\mathcal{W} ^{1,n}\bigg(\A(r, R), \A_\ast(r_*, R_*)\bigg)$ which keep the boundary circles in the same order.
This means that $\lim\limits_{|x|\to r} |h(x)| = r$ and
$\lim\limits_{|x|\to R} |h(x)| = R_\ast$.
We minimize $\mathcal{E}[h]=\mathcal{E}_\Phi[h]$ under the assumption that $h$ is a certain homeomorphism that belongs to the class $\mathcal{R}(\A,\A_\ast)$ for $n\ge 2$.
Furthermore, we revise the same problem considered in \cite{annalen} for $n=2$. In fact, we remove the assumption that $\ddot\Phi(0)=\infty$ (cf. \cite[Theorem 1]{annalen}) and consider more general class $\mathcal{F}(\A,\A_\ast)$. Our main result is as follows.
\begin{theorem}\label{omainth}
Let $\Phi\in C^3\big(0, \infty\big)$ be a positive and convex function. Assume also that $\chi=1/\ddot\Phi$ and its derivative extends continuously on $[0,\infty)$ with $\chi(0)=\alpha\ge 0$.  Let $R>r$. Then for $\alpha>0$ there exists a constant $R_\circ\in(r,R)$  that depends on $\Phi$ and $r$ and $R$ such that  the boundary value problem
$$\left\{
  \begin{array}{ll}
    \ddot H(t)=(H-t\dot H)M(t), & \hbox{\text{where $M$ is defined in \eqref{mmm}} ;} \\
    H(r)=r_\ast\ \ \  H(R)=R_\ast & \hbox{}
  \end{array}
\right.$$
admits a unique solution $H=H_{\lambda_\ast}\in C^\infty(a,\infty)$, where $a<r_\ast$ and $\dot H>0$, if and only if $R_\ast\ge R_\circ$. For $\alpha=0$ we have $R_\circ =r_\ast$ and $R_\ast>1$ is arbitrary.

Moreover, for $n\ge 2$ the radial mapping
$h=h_{\lambda_*}$, defined by
$h(x)=H_{\lambda_*}(|x|)\frac{x}{|x|}$,  minimizes the energy function
$\mathcal{E}_\Phi:\mathcal{R}(A,A_*)\to \mathbf{R}$.
In particularly, if $n= 2$, then $h$ minimizes the energy function
$\mathcal{E}_\Phi:\mathcal{F}(A,A_*)\to \mathbf{R}$.
\end{theorem}

\begin{remark}
If $r<R$ and $r_\ast<R_\ast$, then
the mapping $f\in \mathcal{F}\left( \A(r,R), \A(r_\ast, R_\ast)\right)$ if and only if the mapping $\frac{1}{r_\ast}f\left(\frac{x}{r}\right)$ belongs to the class $\mathcal{F}\left(\A(1,\rho),\A(1,\rho_\ast )\right)$, where $$\rho=\frac{R}{r}>1,\ \ \  \rho_\ast =\frac{R_\ast}{r_\ast}>1.$$
Without loss of generality, we can assume that $r_\ast=r=1$ and $R_\ast,R>1$. So our main result can be formulated as follows:
\end{remark}
\begin{theorem}\label{mainth}
Let $\Phi\in C^3(0, \infty)$ be a positive and convex function.
Assume also that $\chi=1/\ddot\Phi$ and its derivative extends continuously on $[0,\infty)$ with $\chi(0)=\alpha\ge 0$.
Let $R>1$. Then for $\alpha>0$ there is $R_\circ\in(1, R)$  which depends on $\Phi$, $R$ such that  the boundary value problem
\begin{equation}\label{zhu-Oct-26-1}\left\{
  \begin{array}{ll}
    \ddot H(t)=(H-t\dot H)M(t), & \hbox{\text{where $M$ is defined in \eqref{mmm}} ;} \\
    H(1)=1\ \ \  H(R)=R_\ast & \hbox{}
  \end{array}
\right.\end{equation}
admits a unique solution $H\in C^\infty(a,\infty)$ where $a<1$ and $\dot H>0$,
if and only if $R_\ast\ge R_\circ$.

In particular, if $\alpha=0$, then we have $R_\circ =1$ and $R_\ast>1$ is arbitrary.

Moreover, for $R_*\geq R_\circ$ there exist a unique $\lambda\in[0, \infty)$ such that the solution $H$
of $(\ref{zhu-Oct-26-1})$ satisfies the condition $\dot H(1)=\lambda$. We denote that solution by $H_\lambda$.
\end{theorem}

For $R>1$ and $R_\ast\ge R_\circ$ as in Theoren~\ref{mainth}, let $\A=\A(1,R)$ and $\A_\ast=\A_\ast(1,R_\ast)$.
Now we formulate the following result.
\begin{theorem}\label{oct-22-Thm-3}
Under the conditions of Theorem~$\ref{mainth}$, the radial mapping
$h=h_{\lambda}$ defined by
$h(x)=H_{\lambda}(|x|)\frac{x}{|x|}$,  minimizes the function
$\mathcal{E}_\Phi:\mathcal{R}(\A,\A_*)\to \mathbf{R}$.
\end{theorem}

For $n=2$, Iwaniec and Onninen (cf. \cite{annalen}) considered the function
$\Phi \in C^\infty(0,\infty)$ which is positive and strictly convex.
Moreover, they assumed the following condition on $\Phi$:
\emph{The function $\Psi(s)=1/\ddot \Phi(s)$ and its derivative extend continuously to $[0,\infty)$, with $\Psi(0)=0.$}
We are interested in the case when $\Phi$ does not satisfies this condition.
We have the following counterpart of their result in \cite[Theorem 1]{annalen}.

\begin{theorem}\label{thew}
Let $n= 2$. Assume that $\Phi$,  $R>1$, $\lambda$ and $R_\circ$ are defined as in Theorem $\ref{mainth}$. Assume that $\mathcal{F}(A,A_*)$ is the family of of  all  orientation  preserving homeomorphisms
 $h  : \A \mapsto \A_\ast$ in the Sobolev space $\mathcal{W} ^{1,2}(\A, \A_\ast)$ which keep the boundary circles in the same order  with finite energy. Then there exists a diffeomorphism that minimizes the function
$\mathcal{E}_\Phi:\mathcal{F}(A,A_*)\to \mathbf{R}$ if and only if  If $R_*\ge R_\circ$.  In this case the minimizer is the radial mapping
$h=h_{\lambda}$ defined as $h(x)=H_{\lambda}(|x|)\frac{x}{|x|}$.
\end{theorem}

The paper is organized as follows, in Section \ref{sec-2} we give some results on Neohookean energy of radial mappings in higher dimensional.
In Section \ref{sec-3} we present the proofs of Theorem \ref{mainth}--Theorem \ref{thew}.
In Section \ref{sec-4} we determine $R_\circ$ in the quadratic case and give the application of comparison theorem.

\section{Auxiliary results}\label{sec-2}
\subsection{Hilbert norm of derivatives of the radial stretching and the Jacobian}
Assume that $h(x)=H(s)\frac{x}{s}$, where $s=|x|$. Let
$\mathcal{H}(s):=\frac{H(s)}{s}$. Since $\mathrm{grad}
(s)=\frac{x}{|x|}$, we obtain
$$Dh(x)=\left(\mathcal{H}(s)\right)'\frac{x\otimes   x}{s}+ \mathcal{H}(s) \mathbf{I},$$ where $\mathbf{I}$ is the identity matrix.
For $x\in {\mathbb A}$, let $T_1=N=\frac{x}{|x|}$. Further, let
$T_2$, $\dots$, $T_n$ be $n-1$ unit vectors mutually orthogonal and
orthogonal to $N$. Then
\[\begin{split} \|Dh(x)\|^2&=\sum_{i=1}^n|Dh(x) T_i |^2=\sum_{i=1}^n\left|\left(\mathcal{H}(s)\right)' \frac{\left<T_i,x\right>}{s} x +\mathcal{H}(s) T_i \right|^2
\\&
= \left(\mathcal{H}'(s)\right)^2  s^2+n \left( \mathcal{H}(s)\right)^2+2\mathcal{H}(s)\mathcal{H}'(s) s
 \\&=\frac{n-1}{s^2} H^2+\dot H^2.\end{split}\]
Moreover, with respect to the basis $T_i$ ($i=1,\dots, n$), we have
$$D^*h Dh=\left(
                                   \begin{array}{cccc}
                                     \dot H^2 & 0 & \dots & 0 \\
                                     0 & \frac{H^2}{s^2} & \dots & 0 \\
                                     \vdots &  \vdots & \ddots & \vdots \\
                                     0 & 0 & \dots & \frac{H^2}{s^2} \\
                                   \end{array}
                                 \right),$$ where $D^*h$ is the adjugate of the matrix $Dh$.
Thus $$J=J_h=\frac{\dot HH^{n-1}}{s^{n-1}}.$$


\subsection{Neohookean energy of radial mappings}\label{sec2-2}
The Neohookean energy of a mapping $f\in \mathcal{F}(\A,\A_\ast)$ is defined by
\[\begin{split}\mathcal{E}_\Phi[h]=\int_{\mathbb{A}} \|Dh\|^n +\Phi(\mathrm{Det}\, Dh).\end{split}\]
Then we should calculate the energy of a radial stretching $h=H(r)\frac{x}{|x|}$.

The Euler Lagrange equation is given as follows
$$\Lambda_H-\partial_t \Lambda_{\dot H}=0,$$
where
\begin{equation}\label{zhu-Oct-25-1}\Lambda(t,H,\dot H)= t^{n-1}\left(\|Dh\|^n +\Phi(\mathrm{Det}\, Dh)\right),\end{equation}
and
$\Lambda_H=\frac{\partial \Lambda}{\partial H}$, $\Lambda_{\dot H}=\frac{\partial\Lambda}{\partial \dot H}$.

For
$$\|Dh\|^n= \left(\frac{n-1}{t^2} H^2+\dot H^2\right)^{n/2}$$
and
$$\mathrm{Det}\, Dh= \frac{\dot HH^{n-1}}{t^{n-1}},$$
after long but straightforward calculations, the Euler Lagrange equation reduces to
\begin{equation}\label{hsec}\ddot H(t)= M(t)(H-t\dot H),\end{equation}
where
\begin{equation}\label{mmm}
M(t)=\frac{n-1}{t^2 H}\frac{B+C}{D+E}.
\end{equation}
Here
$$B=n t^{2 n} H^3 \left(\frac{(n-1) H^2}{t^2}+\dot H^2\right)^{n/2} \left((n-1) H^2+(n-2) t H \dot H+t^2 {\dot H}^2\right),$$
$$C= H^{2 n} \dot H \left((n-1) H^2+t^2 \dot H^2\right)^2 \ddot \Phi\left[t^{1-n} H^{n-1} \dot H\right],$$

$$D=(n-1) n t^{2 n} H^2 \left(\frac{(n-1) H^2}{t^2}+\dot H^2\right)^{n/2} \left(H^2+t^2 \dot H^2\right)$$ and $$E=H^{2 n} \left((n-1) H^2+t^2 \dot H^2\right)^2 \ddot \Phi\left[t^{1-n} H^{n-1} \dot H\right].$$
The equality (\ref{hsec}) can be rewritten as follows: $$ (H-s\dot H)'=-s (H-s\dot H) M(s),$$ which is
equivalent to

$$\big[\log(H-s\dot H)\big]'=-s M(s).$$
Thus $$\log(H-s\dot H) = \int_{s_1}^{s} (-\tau M(\tau))d\tau+c,$$
which gives that
\begin{equation}\label{lcon}H-t\dot H=c \exp\left[\int_{t_1}^{t} (-\tau M(\tau))d\tau\right].\end{equation}
If the diffeomorphic solution of \eqref{hsec} exists such that $H(r)=r_\ast$ and $H(R)=R_\ast$\,, then it follows from \eqref{lcon} that the expression $H-t\dot H$ cannot change sign. If it is negative, then $\frac{H(t)}{t}$ is increasing; and if it is positive, then $\frac{H(t)}{t}$ is decreasing. In both cases we chose the substitution $\dot H(t) = \frac{v(s)} {s}$
where $s = \frac{H(t)}{t}$.

We now consider three cases concerning the elasticity function
\begin{equation}\label{inel}\mu_H(t) = \frac{t\dot H}{H}<1, \ddot H>0 \text{ for all $r < t < R$ (inelastic case)}\end{equation}
\begin{equation}\label{conf}\mu_H(t) = \frac{t\dot H}{H}=1, \ddot H=0 \text{ for all $r < t < R$ (conformal  case)}\end{equation}
\begin{equation}\label{inel1}\mu_H(t) = \frac{t\dot H}{H}>1, \ddot H<0 \text{ for all $r < t < R$ (elastic case)}.\end{equation}
Then
\begin{equation*}\dot v=G(v,s):=-\frac{X_n(v,s)}{Z_n(v,s)},\end{equation*}
where \[\begin{split}X_n(v,s)&=(n-1) n s^4 \left((n-1) s^4+v^2\right)^{\frac{n}{2}}\\ &\times \left((n-1) s^6+v ((n-3) s^4+(s^2-v) v)\right)\\&+(n-2) s^{3 n} v \left((n-1) s^4+v^2\right)^2 \ddot\Phi\left(s^{n-2} v\right)\end{split}\]
and
\[\begin{split}Z_n(v,s)&=(n-1) n s^{5} \left(s^4+v^2\right) \left((n-1) s^4+v^2\right)^{n/2}\\&+s^{1+3 n} \left((n-1) s^4+v^2\right)^2 \ddot\Phi\left(s^{n-2} v\right).\end{split}\]
Since \begin{equation}\label{dots}\dot v=-\frac{(n-2) v}{s}+\frac{Y_n(v,t)}{Z_n(v,s)}\end{equation}
where \[\begin{split}Y_n(v,s)&=(n-1) n s^4 \left(v-s^2\right)\\&\times  \left((n-1) s^4+v^2\right)^{n/2} \left((n-1) s^4+(n-2) s^2 v+(n-1) v^2\right)\end{split}\]
we see that
 $$|G(v,s)|\le  \frac{(n-2) v}{s}+\left|{\frac{\left(s^2-v\right) \left((n-1) s^4+(n-2) s^2 v+(n-1) v^2\right)}{s \left(s^4+v^2\right)}}\right|.$$
Thus \begin{equation}\label{bounded}|G(v,s)|\le \frac{(3n-5)v}{s}+(n-1)s.\end{equation}

Furthermore, calculations lead to
\begin{equation}\label{Nov-5-zhu}
\dot v=\dot H-MtH,
\end{equation}
where $M$ is given by (\ref{mmm}). This shows that to solve (\ref{hsec}), one can look for $H$ by solving the equation
$\dot H=\frac{t}{H}v(H/t)$.

\section{Proofs of Theorem \ref{mainth} -- Theorem \ref{thew}}\label{sec-3}
\subsection*{Proof of Theorem \ref{mainth}}
As was shown in Section \ref{sec2-2}, to prove Theorem \ref{mainth} we need to solve the auxiliary equation \begin{equation}\label{vg}\dot v(s)=G(s,v).\end{equation}

According to the assumptions we see that $G\in C((0,+\infty)\times \mathbf{R})$ and $G_v(t,0+)$ exists.
We can define $G$ for $v<\infty$ by setting $G(t,-v)=2G(t,0)-G(t,v)$.
Furthermore, it follows from Picard--Lindel\"of Theorem that for every $(s_\circ,v_\circ)\in  (0,+\infty)\times \mathbf{R}$ there exists a local solution $v$ of \eqref{vg} with the initial condition $v(s_\circ)=v_\circ$.
\begin{description}
  \item[Claim 1]  Let $(a,b)$ be the maximal interval of the solution $v$ containing $s_0$.  Then $a=0$ and $b=\infty$.
\end{description}

Prove that $a>0$ leads to the contradiction. In this case $v$ is bounded near $a$. Assume that it is not bounded, and let $a_k $ be a sequence of points such that $a_k\geq a$
satisfying $\lim\limits_{k\rightarrow\infty}a_k=a$ and $\lim\limits_{k\rightarrow\infty}v(a_k)=\infty$.
Then for large enough $k$, there exists $s\in(a, a_k)$ such that $v(s)\geq a^2$ and $v(a_k)>a$.
Namely, assume that $s\in (a,a_k)$. Then for some $m$, $s\in [a_m,a_{m+1}]$, if the minimum of $v$ in $[a_m,a_{m+1}]$ is not $v(a_m)$, then there exits a $s_\circ\in (a_m,a_{m+1})$ such that $\dot v(s_\circ)=0$. According to $\eqref{dots}$ we see that $v(s_\circ)>s_\circ^2>a^2$.   Thus  from \eqref{bounded}, for $a<s<a_k $ we have that $$\frac{|\dot v(s)|}{|v(s)|}\le \frac{3n-5}{s}+\frac{s(n-1)}{ a^2}.$$ By integrating the previous inequality for $a<s<a_k$ we obtain $$\abs{\log v(a+0)-\log\left(v\left(a_k\right)\right)}\le (3n-5)\log\frac{a_k}{a}+\frac{(n-1)}{a^2}\left(a_k^2-a^2\right).$$

Thus $v(a+0)$ is finite, and we can continue the solution $v$ below $a$. This is a contradiction to the assumption and therefore $a=0$.

Prove that $b<\infty $ leads to the contradiction. In this case we prove $v$ is bounded near $b$.
Assume that it is not bounded, and let $b_k $ be a sequence of points such that
$b_k\leq b$ satisfying $\lim\limits_{k\rightarrow\infty}b_k=b$ and $\lim\limits_{k\rightarrow\infty}v(b_k)=\infty$.
Then for large enough $k$, there exists $s\in(b_k,b)$ such that $v(s)\ge b^2/2$ and $v(b_k)>b$.
Namely, assume that $s\in (b_m,b)$. Then for some $m$, $s\in [b_m,b_{m+1}]$, if the minimum of $v$ in $[b_m,b_{m+1}]$ is not $v(b_m)$, then there exists a $s_\circ\in (b_m,b_{m+1})$ such that $\dot v(s_\circ)=0$. According to $\eqref{dots}$, we see that $v(s_\circ)>s_\circ^2>b_m^2$.   Thus  from \eqref{bounded}, for $b_k<s<b $ we have that $$\frac{|\dot v(s)|}{|v(s)|}\le \frac{3n-5}{s}+\frac{4s(n-1)}{ b^2}.$$ By integrating the previous inequality for $b_k<s<b$ we obtain $$\abs{\log v(b^-)-\log\left(v\left(b_m\right)\right)}\le (3n-5)\log\frac{b}{b_k}+\frac{4(n-1)}{b^2}\left(b^2-b_k^2\right).$$
Thus $v(b^-)$ is finite, and we can continue the solution $v$ above $b$. This is a contradiction to the assumption and therefore, $b=\infty$.
\begin{description}
  \item[Claim 2] Denote by $v=v_\lambda(s)$ the solution of $v$ with the initial condition $v_\lambda(1)=\lambda\in[0,+\infty)$. Then every solution of $v_{\lambda}$ is bounded.
\end{description}
If $v_\lambda(s)>s$ for all $s$, then $$\dot v\le -\frac{(n-3)v}{s}-s.$$ The solution of equation $$u'= -\frac{(n-3)u}{s}-s$$ is $$u(s)=-\frac{s^2}{n-1}+c x^{3-n}.$$
Thus $w=v-u$ is the solution of differential inequality  $$s^{2-n}(s^{n-3} w(s))'=sw'+(n-3)w\le 0.$$
This shows that
$$s^{n-3} w(s)\le w(1)=u(1)-v(1)$$ for $s\ge 1$.

By substitution we get
$$s^{n-3}(u(s)+\frac{s^2}{n-1}+c x^{3-n})\le \lambda+c_1$$
i.e.
$$s^{n-3}(v(s)+\frac{s^2}{n-1})\le c_2.$$
Hence
$$v(s)\le c_2 s^{3-n}-\frac{s^2}{n-1}$$ for $s>1$.
This contradicts the fact that $v(s)>s^2$ for all $s$.

The conclusion is that, there is $s_0$ so that $v(s_o)<s_o^2$. Let $s_\diamond$ be the maximal point such that $v(s_o)<s_o$  in $(s_o,s_\diamond)$. Then $s_\diamond=\infty$. Namely, it follows from \eqref{dots} that $v$ is decreasing in $(s_o,s_\diamond)$, if $s_\diamond<\infty$, then $v(s_\diamond)<v(s_o)<s_o<s_o^2<s_\diamond^2$.
By continuity we can find a point $s_1>s_\diamond$ such that $v(s)<s^2$  in $(s_o,s_1)$. Thus by \eqref{dots}, $v$ is decreasing in $(s_o,\infty)$. We can now conclude that $v$ is bounded on $(\tau,\infty)$ for any positive number $\tau>0$.
\begin{description}
  \item[Claim 3]  The boundary value problem admits a unique solution $H\in C^{\infty}(a, \infty)$ where $a<1$ and $\dot H>0$,
if and only if $R_\ast\ge R_\circ$.
\end{description}
According to the definition of $v_\lambda$ in the former Claim, we see that the mapping \begin{equation}\label{hh}V(s,\lambda)=v_\lambda(s):(0,\infty)\times [0,+\infty)\to [0,+\infty)\end{equation} is of class $\mathcal{C}^ 1(\mathbf{R}_+,\mathbf{ R})$, see \cite[Ch.~V,~ Corollary 4.1]{hartman}.
By using (\ref{Nov-5-zhu}), to solve (\ref{hsec}), one can consider the following differential equation with the initial condition:
\begin{equation}\label{hlam}\left\{
  \begin{array}{ll}
    \dot H(t)=\frac{H(t)}{t}v_\lambda\left(\frac{H(t)}{t}\right), & t\in (\alpha,\beta) \\
    H(1)=1. & \hbox{}
  \end{array}
\right.\end{equation}
The local solution of (\ref{hlam}) exists because of Picard-Lindel\"of Theorem.
Denote by $(\alpha,\beta)$ the maximal interval of the solution of the equation \eqref{hlam}. It is clear that $0\le \alpha<1$.
If $\beta<\infty$, then since $v$ is bounded, it follows from \eqref{hlam} (by integrating in $[1,\beta)$) that $$\log \frac{H(\beta-0)}{H(1)}\le C\log \frac{\beta}{1}$$
and thus $H(\beta-0)$ is finite. Then, we can continue the solution above $\beta$. Therefore, $\beta=\infty$.

Since $v_{o}(s)$ is a particular solution of \eqref{vg}, because $V$ defined in \eqref{hh} depends continuously on $\lambda$, there exists a particular solution $H_\circ$ of \eqref{hlam} with $\lambda=0$. Now a fixed $R$ we define
\begin{equation}\label{ro}R_o=H_\circ (R).\end{equation}

Prove now that, the diffeomorphic solution $H$ exists with $H(1)=1$ and $H(R)=R_\ast$ if and only if $R_\ast\ge R_o$,

For every $\lambda>0$, $v_\lambda(s)>v_{o}(s)$ and for every $\lambda<0$, $v_\lambda(s)<v_{o}(s)$ since \eqref{vg} has unique solution with the initial condition. Furthermore, for $0<\lambda<1$, the solution $H_\lambda$ to the equation \eqref{hlam} satisfies the inequality $1\cdot \dot H(1)-H(1)<0$, and thus inelasticity case \eqref{inel} occur. This implies that for all the points $t\in(a,\infty)$ we have $tH'(t)\le H(t)$ and thus $\frac{H(t)}{t}$ is decreasing. In particular we have
\begin{equation}\label{hoh}\frac{H_o(R)}{R}\le \frac{H_\lambda(R)}{R}<1,\end{equation}
for every $0\le \lambda< 1$. This implies that $$R_\circ<R$$ and $$R_*:=H_\lambda(R)\geq R_\circ.$$ The last inequality, implies that, if there exists a diffeomorphic solution $H:[1,R]\to [1,R_\ast]$ of the differential equation \eqref{hlam}, then $R_\ast\ge R_\circ$.
In order to show that every diffeomorphic solution $H:[1,R]\to [1,R_\ast]$ of the differential equation \eqref{zhu-Oct-26-1} admits the initial condition $\dot H(1)=\lambda$, we proceed as follows. Assume that $H$ is one such solution. Then by \eqref{lcon}, $H(t)/t$ is either increasing or decreasing or a constant function. If it is not a constant, a case which is easy to deal with, then $s=\sigma(t)=\frac{H(t)}{t}$ is a diffeomorphism. Now the function $u(s)$ defined by  $\dot H(\sigma^{-1}(s)) = \frac{u(s)} {s}$ satisfies the differential equation \eqref{dots}. If $u(1)=\lambda$, then $u=v_\lambda$. Further we obtain that $H=H_\lambda$, where $\lambda=\dot H(1)$. And this is all what is needed to prove for this direction.

Further we have $$H_\lambda(R)>R, \ \ \ \lambda>1.$$
Hence if $H=H_\lambda$ is the solution of \eqref{hlam}, then $$\lim_{\lambda\to 0}H_\lambda(R)=H_o(R).$$
Finally, we show that $$\lim_{\lambda\to +\infty} H_\lambda(R)=\infty.$$
Assume that $H_\lambda(R)\le M$ for every $\lambda$. Then we have \begin{equation}\label{baterfly}\log \frac{H(R)}{H(1)}\ge \log{R}\int_1^Rv_\lambda\left(\frac{H(t)}{t}\right)dt.\end{equation}
Thus $s=H(t)/t\in (M/R, M)$. Since $v_\lambda(s)>v_{\lambda_0}(s)$ for every $s$ if $\lambda>\lambda_0$, it follows from \eqref{bounded} that
$$\left|\log \frac{v_\lambda(s)}{v_\lambda(1)}\right|\le (3n-5)\log s +\frac{(s^2-1)(n-1)}{m}=M_0,$$ where $$m=\min\{v_{\lambda_0}(s): M/R\le s\le M/1\}.$$
This shows that
$$v_\lambda(s)\ge e^{-M_0}\lambda.$$
Applying \eqref{baterfly} we conclude that $$H(R)\ge \exp((R-1)e^{-M_0}\lambda \log{R} ).$$
This contradicts with the assumption that $H_\lambda(R)$ is bounded and therefore, $\lim\limits_{\lambda\to +\infty} H_\lambda(R)=\infty.$

Assume now that $R_\ast\ge R_\circ$. We need to show that there is a solution with $H(1)=1$ and $H(R)=R_\ast$. But this follows easily from the fact that the function $P(\lambda)=H_\lambda(R)$ is continuous and $\lim\limits_{\lambda\to 0}H(R)=H_\circ$ and $\lim\limits_{\lambda\to \infty}H(R)=+\infty.$

\subsection*{Proof of Theorem \ref{oct-22-Thm-3}}
Following the proof of Theorem \ref{mainth}, we know that the stationary point of (\ref{hsec}) is unique. We only need to
show that the given energy integral attains its minimum.

According to (\ref{zhu-Oct-25-1}), we see that
$$\Lambda(t,H,\dot H)= t^{n-1}\left(\left(\frac{n-1}{t^2} H^2+\dot H^2\right)^{n/2}+\Phi\left(\frac{\dot HH^{n-1}}{t^{n-1}}\right)\right).$$
For $K=\dot H$, we have the following formula (where $\mathcal{L}[t,H,K]$ is the subintegral expression for $\mathcal{E}[h]$)
\[\begin{split}\partial_{KK}\mathcal{L}[t,H,K]&=\frac{(-1+n) n t^{1+n} \left(\frac{H^2 (-1+n)}{t^2}+\dot{H}^2\right)^{n/2} \left(H^2+t^2 \dot{H}^2\right)}{\left(H^2 (-1+n)+t^2 \dot{H}^2\right)^2}\\&+H^{-2+2 n} t^{1-n} \ddot\Phi \left[H^{-1+n} t^{1-n} \dot{H}\right] \end{split}\]
 which is clearly positive and thus $\Lambda(t,H,\dot H)$ it is convex in $K=\dot H$.

Furthermore, since $r\le t\le R$, we can find a positive constant $C$ such that
\begin{equation}\label{bound1}C|\dot H|^n\le \mathcal{L}[t,H,\dot H].\end{equation} This implies that the function $L$ is coercive.

Let $h_m(x)=H_m[|x|]x/|x|$ be a sequence of smooth mappings  with $H_m(r)=r_\ast$, $H_m(R)=R_\ast$ and $$\inf_{h\in \mathcal{R}(A,A_*)}
  \mathcal{E}[h]=\lim_{m\to \infty}\mathcal{E}[h_m].$$
We have $H_m$ are diffeomorphisms since $\dot H_m\neq 0$. Then up to a subsequence it
   converges to a monotone increasing function $H_\circ$.  Moreover,  since $H_m$ is a bounded sequence of $\mathscr{W}^{1,n}$, it converges,
   up to a subsequence weakly to a mapping $H_\circ \in \mathscr{W}^{1,n}$.

By using the mentioned convexity of $\mathcal{L}$  and the fact that $\mathcal{L}$ is coercive, together with the standard theorem from the calculus of
variation (see \cite[p.~79]{bern}),  we obtain that $$\mathcal{E}[h_\circ]= \lim_{m\to \infty}\mathcal{E}[h_m].$$
Moreover, since $\mathcal{L}[t,H,K]\in C^\infty(\mathbf{R}_+^3)$ and $\partial^2_{KK}\mathcal{L}[t,H,K]>0$, we infer that
$H_\circ\in C^\infty[r,R]$ (see \cite[p.~17]{jost}) and $H_\circ$ is the solution of the Euler-Lagrange equation. Thus it coincides with $H_{\lambda_\ast}$.
\qed

\subsection*{Proof of Theorem \ref{thew}}
By using Theorem \ref{mainth} and the proof of \cite[Theorem~1]{annalen}, one can obtain this result.
\qed
\section{Determination of $R_\circ$ of the quadratic case and its application}\label{sec-4}
In the following part of this section we will determine $R_\circ$ is some special cases.
\subsection{The quadratic case $\Phi_\kappa(t)=\alpha+\beta t+{\kappa^2}t^2$} Then $\Phi$ is convex and satisfying
$$\frac{1}{\ddot\Phi(0)}=\frac{1}{2\kappa^2}\neq 0,$$ which makes complementary to the problem treated in \cite{annalen}.

Then the corresponding Euler equation is \begin{equation}\label{soluniq}\ddot H(t)= \frac{\left(H(t)-t \dot H(t)\right)
\left( t+\kappa^2 H(t) \dot H(t)\right)}{ t^3+\kappa^2 t H(t)^2}.\end{equation}
Let $s=\frac{H(t)}{t}$ and $\dot H=\frac{v(s)}{s}$. The auxiliary equation \eqref{dots} for this special case is
$$\dot v(s)=-\frac{ \left(s^2-v(s)\right)}{s \left(1+\kappa^2 s^2\right)}$$ whose solution is given by
$$ v(s)=\frac{s \left( c_0- \sinh^{-1}\left[{\kappa s}{}\right]\right)}{\kappa \sqrt{1+\kappa^2 s^2}},$$
where $c_0$ is a constant depends on $H$.
Now we assume that $\dot H(t)\ge 0$, and this implies that \begin{equation}\label{con}c_0\ge   \sinh^{-1}\left[{\kappa}{}\right].\end{equation}

Assume further that $1<R$ and $1<R_\ast$.
The unique solution $H$ to the equation \eqref{soluniq}  is given implicitly by
 \begin{equation}\label{equati} \log t+\frac{1}{2}\log\left|A\right|=c_1,\end{equation} where
$$A=2  \sinh^{-1}\left[\frac{\kappa H[t]}{ t}\right]-2 c_0 +2\kappa \frac{H[t]}{t} \sqrt{1+\frac{\kappa^2 H[t]^2}{t^2}},$$
and $c_1$ is a constant depends on $R$ and $R_*$.
Now taking into account the condition  $H(1)=1$ and $H(R)=R_\ast$ we get
\begin{equation}\label{coo}c_0=\frac{-\kappa \sqrt{1+\kappa^2}+\kappa  R_\ast \sqrt{{R^2+\kappa^2 R_\ast^2}}-\sinh^{-1}[\kappa]+r^2 \sinh^{-1}\left[\frac{\kappa R_\ast}{R}\right]}{R^2-1}\end{equation} and
$$c_1=\frac{1}{2} \log\left|B\right|,$$ where
$$B=\frac{2R \left(\kappa \left(\sqrt{1+\kappa^2} R-R_\ast \sqrt{1+\frac{\kappa^2 R_\ast^2}{R^2}}\right)+
R \sinh^{-1}\left[{\kappa}{}\right]-R \sinh^{-1}\left[\frac{\kappa R_\ast}{ R}\right]\right)}{-1+R^2}.$$
Then
$$|A|=\left\{
        \begin{array}{ll}
          -A, & \hbox{if $R_\ast\ge R$;} \\
          A, & \hbox{if $R_\ast<R$.}
        \end{array}
      \right.
$$
and
$$|B|=\left\{
        \begin{array}{ll}
          -B, & \hbox{if $R_\ast\ge R$;} \\
          B, & \hbox{if $R_\ast<R$.}
        \end{array}
      \right.
$$
since $$\kappa \left(\sqrt{1+\kappa^2}-a \sqrt{1+a^2 \kappa^2}\right)+ \sinh^{-1}\left[{\kappa}{}\right]-\sinh^{-1}\left[{a \kappa}{}\right]\ge 0 $$ if and only if $a=\frac{R_\ast}{R}\le 1.$

Thus \eqref{equati} can be written as
\begin{equation}\label{spsp}\begin{split}&    \sinh^{-1}\left[\frac{\kappa H[t]}{ t}\right]+
2{\kappa} \frac{H[t]}{t} \sqrt{1+\frac{\kappa^2 H[t]^2}{t^2}}=\Psi(t):=\frac{U+V}{(R^2-1)t}\end{split},\end{equation}
where $$U=\kappa \left(\sqrt{1+\kappa^2} R^2-\sqrt{1+\kappa^2} t^2+R R_\ast \sqrt{1+\frac{\kappa^2 R_\ast^2}{R^2}} \left(-1+t^2\right)\right)$$
and
$$V= (R-t) (R+t) \sinh^{-1}\left[\kappa\right]+2 R^2 \left(-1+t^2\right) \sinh^{-1}\left[\frac{\kappa R_\ast}{ R}\right].$$
Let $L(t)=\frac{H(t)}{t}$ and let $\psi(x)=\kappa\text{  }x \sqrt{1+\kappa^2 x^2}+\text{  }\sinh^{-1}\left[\kappa x\right]$.
According to \eqref{spsp}, we see that
\[\begin{split}&L'(t)\psi'(L(x))=\Psi'(t)\\&=\frac{ R \left(2\kappa \left(R_\ast \sqrt{1+\frac{\kappa^2 R_\ast^2}{R^2}}\right)-
 R \sinh^{-1}\left[\kappa\right]-\sqrt{1+\kappa^2} R+2 R \sinh^{-1}\left[\frac{\kappa R_\ast}{ R}\right]\right)}{\left(R^2-1\right) t^3}.\end{split}\]
Since $\psi'(x)>0$, it follows that $L'(t)\ge 0$ if and only if
$$\kappa \left(-\sqrt{1+\kappa^2}+a \sqrt{1+a^2 \kappa^2}\right)-\sinh^{-1}\left[\kappa\right]+
 \sinh^{-1}\left[{a \kappa}{}\right]\ge 0,$$ where $a=\frac{R_\ast}{R}$. Thus $L'(t)\ge 0$
for $1\le t\le R$ if and only if $a\ge 1$ (i.e. $R_\ast\ge R$).
Similarly, $L'(t)<0$ iff $a<1$ (i.e. $R_\ast<R$).
In both cases the diffeomorphic solution is $$H(t)=t\cdot \psi^{-1}(\Psi(t)).$$
Now the condition \eqref{con}, in view of \eqref{coo} can be written as
$$\frac{-\kappa \sqrt{1+\kappa^2}+\kappa R R_\ast \sqrt{1+\frac{\kappa^2 R_\ast^2}{R^2}}-\sinh^{-1}\left[\kappa\right]+
R^2 \sinh^{-1}\left[\frac{\kappa R_\ast}{ R}\right]}{-1+R^2}\geq  \sinh^{-1}\left[\kappa\right]$$
or what is the same
\begin{equation}\label{davidn12}{ \frac{ R_\ast}{R} \sqrt{1+\kappa^2\frac{ R_\ast^2}{R^2}}-\frac{\sqrt{1+\kappa^2}}{R^2}\ge\frac{1}{\kappa} \left(\sinh^{-1}\left[\kappa\right]-
\sinh^{-1}\left[\frac{\kappa R_\ast}{ R}\right]\right)}.
\end{equation}
Thus  \begin{equation}\label{rcirci}R_\circ=R_\circ(R,\kappa)\end{equation} is the unique solution of the equation
\begin{equation}\label{davidn13}{ \frac{ R_\circ}{R} \sqrt{1+\kappa^2\frac{ R_\circ^2}{R^2}}-\frac{\sqrt{1+\kappa^2}}{R^2}=\frac{1}{\kappa} \left(\sinh^{-1}\left[\kappa\right]-
\sinh^{-1}\left[\frac{\kappa R_\circ}{ R}\right]\right)}.
\end{equation}
Let $\kappa$ tends to $0^+$. By using \eqref{davidn12}, we obtain the inequality
$$R_\ast\ge R_\circ=\frac{1+R^2}{2R}$$ which is the standard Nitsche inequality.

Let $\kappa$ tends to $+\infty$. Then \eqref{davidn12} reduces to the inequality $$R_\ast\ge R_\circ=1,$$ which means that no Nitsche phenomenon occurs (as was shown in \cite{annalen}).

Based on the above facts, we can obtain the following Proposition \ref{Zhu-Oct-30}.
\begin{proposition}\label{Zhu-Oct-30}
Let $\kappa>0$.  Then for every $R_\ast,R>1$ there is a diffeomorphic increasing solution $H$ of the ODE \eqref{soluniq} that maps $[1,R]$ onto $[1,R_\ast]$ if and only if $R_\ast$ and $R$ satisfy \eqref{davidn12}. In particular, for every $R_\ast\ge R$, \eqref{davidn12} holds true.
\end{proposition}
By using Theorem \ref{thew}, we have
\begin{theorem}
The minimum of energy $\mathcal{E}_\kappa:=\mathcal{E}_{\Phi_\kappa}$ within the class $\mathcal{F}(\A,\A_\ast)$ is attained for a radial map $h(x)=H(|x|)\frac{x}{|x|}$ if and only if   \begin{equation}\label{davidn21}R_\ast\ge R_\circ.
\end{equation}
The minimum is unique up to the rotation of annuli.
\end{theorem}

The following figure shows the graphic of the function $R_\circ=R_\circ(R)$ for different $\kappa$

\begin{figure}[htp]\label{f1}
\centering
\includegraphics{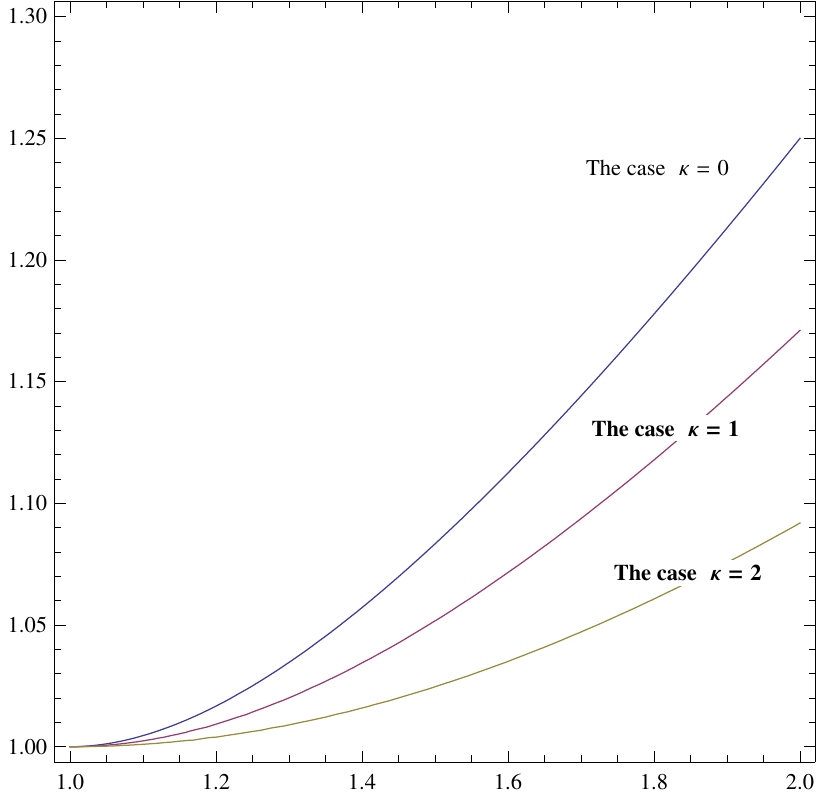}
\caption{The graph of $R_\circ=R_\circ(R)$ tends to the line $R_\circ=1$ when $\kappa\to \infty$}
\end{figure}

\subsection{Application of comparison theorem}
\begin{lemma}\label{mica}
Assume that
$$\dot u(s)=A(u(s),s) \ \ \ \text{and } \ \ \dot v(s)\le A(v(s),s).$$
 Then \begin{equation}\label{important}v(s)\ge u(s)  \ \ \ \ \text{for} \ \  s<1 \ \ \text{and}\ \ \  v(1)\ge u(1)\end{equation} and   \begin{equation}\label{important1}v(s)\le u(s)  \ \ \ \ \text{for} \ \  s>1 \ \ \text{and}\ \ \  v(1)\le u(1)\end{equation} provided that $A$ is Lipschitz continuous.
\end{lemma}
\begin{proof} Let $w(s)=u(s)-v(s)$ and assume that there exists some constant $y<1$ such that $w(y)>0$. Moreover by using the continuity of $w$, we see that there exists a constant $\rho\in[y,1]$ such that $w(\rho)=0$ and $w(x)>0$ for $x\in[y,\rho)$. Let $$\rho=\inf\{d\in [y,1]:w|_{(y,d)}>0\}.$$

Then  for $x\in [y,\rho]$, one has
$$-w'(x)\le A(v(y),y)-A(u(y),y)\le L|v(y)-u(y)|=-Lw(x).$$
Integration on $[y,\sigma]\subset [y,\rho]$ gives

$$-\int_y^{\sigma}\frac{w'(s)}{w(s)}ds\le -L\int_y^{\sigma}ds.$$

This shows that $$\log \frac{w(y)}{w(\sigma)}\le L(y-\sigma)$$ and thus $$ w(y)\le w(\sigma)e^{L(y-\sigma)}.$$
By letting $\sigma$ close to $\rho$ increasingly, one has
 $$ w(y)\le w(\rho)e^{L(y-\rho)}=0$$ which is a contradiction to the condition $w(x)>0$ for $x\in[y, \rho)$.
\end{proof}

Now we are ready to formulate the following proposition.
\begin{proposition}\label{prop}
Assume that $\Phi_1$, $\Phi_2\in C^3(0, \infty)$ are two positive  convex functions which map $(0,\infty)$ into itself.
Assume further that $\ddot\Phi_1(a)\ge \ddot\Phi_2(a)$ for every $a$. If $H_1$ and $H_2$ are solutions of the following ODE
\begin{equation}\label{hdot}\ddot H_i=(H_i-t\dot H_i)\frac{H_i\dot H_i\ddot \Phi_i+2t}{t H_i^2 \ddot \Phi_i+2t^3}, \ \ i=1,2\end{equation}
satisfying the conditions  $H_1(1)=H_2(1), 0<H_1'(1)\le H_2'(1)$, then $$H_1(t)\le H_2(t).$$
\end{proposition}
\begin{proof}

The auxiliary equation is \begin{equation}\label{soleq1}\dot v_i=\frac{2v_i-2s^2}{s^3 \ddot \Phi_i(v_i)+2s}=A_i(v_i(s),s), \ \ i=1,2.\end{equation}
Then
$$A_2(u,s)\le A_1(u,s)$$
which shows that
$$\dot v_2=A_2(v_2(s),s)\le A_1(v_2(s),s).$$
Hence the solutions $v_i$, $i=1,2$ to the equation \eqref{soleq1} according to Lemma~\ref{mica} satisfying the inequality
$$v_1(s)\le v_2(s).$$
Moreover, we have
$$\dot H_i(t)=\frac{t}{H_i(t)}{v_i\left(\frac{H_i(t)}{t}\right)}=\chi_i(t,H_i(t)),  \ \ i=1,2,$$
and by using again Lemma~\ref{mica}, in view of
$$\chi_2(x,y)=\frac{x}{y}v_2\left(\frac{y}{x}\right)\ge \chi_1(x,y)=\frac{x}{y}v_1\left(\frac{y}{x}\right)$$
one has
$$H_1(t)\leq H_2(t).$$
\end{proof}


\begin{corollary}
Assume that $2\kappa^2=\sup\left\{\ddot\Phi(a):a\in(0,\infty)\right\}<\infty$ and assume that there exists an increasing diffeomorphism $H:[1,R]\to [1,R_\ast]$ which solves the ODE \eqref{hdot}.
Then $R_\ast, R$ and $\kappa$ are parameters satisfying the condition \eqref{davidn12}.
\end{corollary}
\begin{proof}

Let $H_\kappa: [1,R]\to [1,R_\ast']$ be a diffeomorphism, where $R_\ast'=R_\ast'(R,\kappa)$ satisfies the equality  \eqref{davidn12}.
Since$$\ddot\Phi(a)\le \ddot\Phi_\kappa(a)=2\kappa^2, $$ it follows from Proposition~\ref{prop}, that $H(t)\ge H_\kappa(t)$ for every $t$. In particular we have the inequality $R_\ast\ge R_\ast'$, and
this implies that the triple $(R_\ast,R,\kappa)$ satisfies the inequality \eqref{davidn12}.
\end{proof}

\vspace*{5mm}
\noindent {\bf Acknowledgments}. The research of the
second author was supported by NSFs of China (No. 11501220), NSFs of Fujian Province (No. 2016J01020),
the Promotion Program for Young and Middle-aged Teacher in Science and Technology Research of Huaqiao University (ZQN-PY402)
and the Program for Innovative Research Team in Science and Technology in Fujian Province University, and Quanzhou High-Level Talents Support Plan under Grant 2017ZT012.

\end{document}